\documentclass[12pt]{smfart}
\usepackage{color}
\usepackage{amssymb,verbatim}
\usepackage{amsthm,array,amssymb,amscd,amsfonts,amssymb,latexsym, url}
\usepackage{amsmath}
\usepackage[all]{xy}
\usepackage[french]{babel}
\usepackage{times}

\definecolor{labelkey}{rgb}{1,0,0}

\newcounter{spec}
{\end{list}}

\renewcommand{\P}{{\mathbf P}}

\newcommand{\Z}{{\mathbb Z}}
\newcommand{\Q}{{\mathbb Q}}

\renewcommand{\L}{{\mathbb L}}
\newcommand{\oi}{\hskip1mm {\buildrel \simeq \over \rightarrow} \hskip1mm}

\newcommand{\ovk}{{\overline k}}

\newcommand{\ovX}{{\overline X}}

\newcommand{\Br}{{\operatorname{Br}}}

\newcommand{\by}[1]{\overset{#1}{\longrightarrow}}
\newcommand{\iso}{\by{\sim}}
\renewcommand{\lim}{\varprojlim}

\renewcommand{\phi}{\varphi}

\numberwithin{equation}{section}

\newfont{\gothic}{eufb10}

\renewcommand{\qed}{{\hfill$\square$}}

\newtheorem{prop}{Proposition}[section]
\newtheorem{thm}[prop]{Th\'eor\`eme}
\newtheorem{lem}[prop]{Lemme}
\newtheorem{cor}[prop]{Corollaire}

\newtheorem{rem}[prop]{Remarque}

\newcommand{\bthe}{\begin{theo}}
\newcommand{\ble}{\begin{lem}}
\newcommand{\bpr}{\begin{prop}}
\newcommand{\bco}{\begin{cor}}
\newcommand{\bde}{\begin{defi}}
\newcommand{\ethe}{\end{theo}}
\newcommand{\ele}{\end{lem}}
\newcommand{\epr}{\end{prop}}
\newcommand{\eco}{\end{cor}}
\newcommand{\ede}{\end{defi}}

\newcommand{\Gal}{{\rm{Gal}}}

\newcommand{\F}{{\mathbb F}}

 \def\br{{\rm Br}}
 \def\pic{{\rm Pic}}
 \def\ov{\overline}
\def\gal{{\rm Gal}}
\def\A{{\bf A}}
\def\G{{\bf G}}

\DeclareFontFamily{U}{wncy}{}
\DeclareFontShape{U}{wncy}{m}{n}{%
<5>wncyr5%
<6>wncyr6%
<7>wncyr7%
<8>wncyr8%
<9>wncyr9%
<10>wncyr10%
<11>wncyr10%
<12>wncyr6%
<14>wncyr7%
<17>wncyr8%
<20>wncyr10%
<25>wncyr10}{}
\DeclareMathAlphabet{\cyr}{U}{wncy}{m}{n}

\begin{document}

\title[Groupe de Brauer non ramifi\'e d'espaces homog\`enes de tores]{Groupe de Brauer non ramifi\'e d'espaces homog\`enes de tores}
\author{Jean-Louis Colliot-Th\'el\`ene}
\address{C.N.R.S., Universit\'e Paris Sud\\Math\'ematiques, B\^atiment 425\\91405 Orsay Cedex\\France}
\email{jlct@math.u-psud.fr}
\date{12 octobre 2012}

\maketitle



\begin{altabstract}
 Let $k$ be a field, $X$ a smooth, projective $k$-variety. If $X$ is geometrically rational, 
there is an injective map from the quotient of Brauer groups $\Br(X)/\Br(k)$ into the first Galois cohomology
group of the lattice given by the geometric Picard group. In this note, where the main attention is
on smooth compactifications of homogeneous spaces of algebraic $k$-tori, we show how under some hypotheses
 the map is onto, and how one may in some special case  exhibit concrete generators in $\Br(X)$.
This is applied to the analysis of counter\-examples to the local-global principle for norms in 
 biquadratic extensions of number fields.
\end{altabstract}
 
 
 \section*{Introduction}
 
 Soient $k$ un corps de caract\'eristique z\'ero,
  $\ovk$ une cl\^oture alg\'ebrique de $k$
et $g=\Gal(\ovk/k)$.
 Soient $X$ une $k$-vari\'et\'e projective, lisse, g\'eom\'etriquement connexe
et  $\ovX=X\times_{k}\ovk$.

On s'int\'eresse au groupe de Brauer $\br(X)$, qu'on peut aussi d\'efinir
comme le groupe de Brauer non ramifi\'e $\br_{nr}(k(X))=\br_{nr}(k(X)/k)$ du corps
des fonctions de $X$. 

Le groupe de Brauer alg\'ebrique $\br_{a}(X)$ est par d\'efinition le noyau
de la restriction
$\br X \to \br(\ovX)$.
Si $\ovX$ est $\ovk$-birationnelle \`a un espace projectif, alors $\br(\ovX)=0$,
et $\br_{a}(X)= \br(X)$.

On sait (voir \cite[(1.5.0)]{descenteII}) que le groupe $\br_{a}(X)$ s'ins\`ere dans une suite exacte 
$$0 \to  \pic(X) \to \pic(\ovX)^g \to \br(k) \to \br_{a}(X) \to H^1(g,\pic(\ovX)) \to H^3(k,\G_{m}).$$
D\'eterminer si une classe dans $H^1(g,\pic(\ovX)) $ se rel\`eve dans
$\br_{a}(X) $ est un premier probl\`eme. Quand on sait qu'il existe un rel\`evement,
exhiber un \'el\'ement concret dans $\br_{a}(X) \subset \br(X) \subset \br(k(X))$ relevant la classe,
et permettant des calculs num\'eriques,
peut \^etre un probl\`eme, quand tout ce que l'on  conna\^{\i}t explicitement est
un ouvert non vide   de $X$,
mais pas la $k$-vari\'et\'e $X$ elle-m\^{e}me.

Si $X$ poss\`ede un point $k$-rationnel,  la suite ci-dessus
 donne un isomorphisme
 $$  \pic(X) \oi \pic(\ovX)^g$$
et  une suite exacte courte
$$ 0 \to  \br(k) \to \br_{a}(X) \to H^1(g,\pic(\ovX)) \to 0.$$
 M\^eme dans ce cas, et m\^eme si l'on s'est donn\'e
 le $k$-point rationnel, si l'on ne conna\^{i}t pas $X$ de fa\c con explicite,
trouver un rel\`evement concret d'un \'el\'ement
 de $H^1(g,\pic(\ovX))$ peut aussi  \^{e}tre un probl\`eme.

\medskip

Un exemple   d'une telle situation se pr\'esente lorsque
l'on \'etudie les compactifications lisses $T^c$, resp. $E^c$,
d'un $k$-tore $T$, resp. d'un espace homog\`ene principal
$E$ de $T$.

\medskip

Un cas typique est le suivant (voir \cite{casselsfrohlich}, \cite{sansuc}, \cite{wei}).

On consid\`ere une extension biquadratique $K/k$, un \'el\'ement $\gamma \in k^*$
et la $k$-vari\'et\'e $E$ d\'efinie par  l'\'equation
$$N_{K/k}(\Xi)=\gamma,$$
o\`u $N_{K/k}$ d\'esigne la norme de $K$ \`a $k$.
Pour $\gamma=1$,  ceci d\'efinit  le $k$-tore $T=R^1_{K/k}\G_{m}$. On sait que l'on a $\br(T^c)/\br(k)=\Z/2$.
Mais on ne sait pas \'ecrire de fa\c con naturelle  un g\'en\'erateur dans $\br(T^c)$
(voir la remarque \ref{difficilebiquad}).

Pour $\gamma$ quelconque, on a $\br(E^c)/\br(k) \subset \Z/2$, 
et on n'a pas de fa\c con syst\'e\-ma\-tique pour d\'ecider si le quotient vaut $\Z/2$ et si c'est le cas trouver un repr\'esentant du  g\'en\'erateur dans $\br(E^c)$.  

\medskip

Le but de cette note est de montrer comment dans certains cas on peut contourner
ces probl\`emes.

Le cas consid\'er\'e aux  paragraphes 4 et 5 permet \`a de la Bret\`eche et
Browning \cite{bretechebrowning} de donner des estimations asymptotiques
sur le nombre de contre-exemples au principe de Hasse.

\medskip

Dans tout l'article, $k$ est un corps de caract\'eristique z\'ero,  $\ovk$ est une cl\^oture alg\'ebrique de $k$
et $g=\Gal(\ovk/k)$.

\section{Vari\'et\'es qui deviennent rationnelles apr\`es une extension cyclique du corps de base}

Dans cette section nous d\'ecrivons un cas o\`u l'application
$$ \br_{a}(X) \to H^1(k,\pic(\ovX))$$ est surjective, que $X$ ait un point rationnel ou non.
On dit qu'une $k$-vari\'et\'e int\`egre est $k$-rationnelle si son corps des fonctions
est transcendant pur sur $k$.

\begin{prop}\label{cyclique}
Soit  $X$ une $k$-vari\'et\'e projective lisse g\'eom\'etriquement connexe.
Supposons qu'il existe une extension finie cyclique $K/k$  et une $K$-vari\'et\'e int\`egre $Y$
 telle que
   $Y \times_{K}X_{K}$ soit une $K$-vari\'et\'e $K$-rationnelle.
Alors :

(a)  On a une suite exacte
$$ \br(k) \to \br(X) \to H^1(k,\pic(\ovX)) \to 0.$$

(b) On a une suite exacte
$$Ker[ \br(k)\to \br(K)] \to \ker[ \br(X) \to \br(X_{K})] \to H^1(k,\pic(\ovX)) \to 0.$$

(c)
Tout \'el\'ement de 
$H^1(k,\pic(\ovX)) $ se rel\`eve en un \'el\'ement du sous-groupe
$\br(X) \subset \br(k(X))$
 qui s'\'ecrit $(\chi,f)$ avec $$f \in \ker [ k(X)^{\times}/N_{K/k}(K(X)^{\times}) \to  {\rm Div}(X)/N_{K/k}({\rm Div} X_{K})]$$
et $\chi \in H^1(K/k,\Q/\Z) \subset H^2(k,\Z)$.
\end{prop}

\begin{proof} Soit $G=\gal(K/k)$. Fixons un plongement $K \subset \ov{k}$,
d'o\`u une application quotient surjective $g \to G$.
On compare les suites exactes d\'eduites des suites spectrales
$$E_{2}^{pq} = H^p(G, H^q(X_{K},\G_{m})) \Longrightarrow H^n(X,\G_{m})$$
et
$$E_{2}^{pq} = H^p(g, H^q(\ovX,\G_{m})) \Longrightarrow H^n(X,\G_{m}).$$

On a donc le diagramme commutatif de suites exactes
 \begin{equation}
 \begin{CD}
\ker[\br(k) \to \br(K)] & \to  & \ker[\br(X) \to \br(X_{K})]  & \to  & H^1(G,\pic(X_{K}))& \to  & H^3(G,K^{\times}) \\
  \downarrow&&  \downarrow &  &\downarrow &   &\downarrow &  &  \\
  \ker[\br(k) \to \br(\ovk)] & \to  &  \ker[\br(X) \to \br(\ovX)] & \to  &  H^1(g,\pic(\ovX)) & \to  &  H^3(g,\ovk^{\times})
 \end{CD}
\end{equation} 
 Comme $G$ est cyclique, on a $H^3(G,K^{\times}) \simeq H^1(G,K^{\times})$ et ce dernier
groupe est nul d'apr\`es le th\'eor\`eme 90 de Hilbert.

Soit $h=\gal({\ovk}/K)$. L'hypoth\`ese faite sur  la $K$-vari\'et\'e $X_{K}$ implique qu'elle poss\`ede un
$K$-point et que le $h$-module $\pic(\ovX)$ est un facteur direct d'un module de permutation  (\cite[Prop. 2.A.1]{descenteII}).
Le premier fait implique $\pic(X_{K}) \oi \pic(\ovX)^h$. Le second implique $H^1(h,\pic(\ovX) )=0$.
La suite exacte naturelle
$$0 \to H^1(G,\pic(\ovX)^h) \to H^1(g,\pic(\ovX) )\to H^1(h,\pic(\ovX) )$$
devient un isomorphisme
$$ H^1(G,\pic(X_{K})) \oi H^1(g,\pic(\ovX)).$$
L'hypoth\`ese faite sur $X_{K}$ implique qu'il existe un entier $n>0$ et une $\ov{k}$-application rationnelle dominante
de $\P^{n}_{\ov{k}}$ vers $\ov{X}$ qui admet une section sur un ouvert de $\ov{X}$. Ceci suffit \`a assurer $\br(\ov{X})=0$.
On a ainsi \'etabli les \'enonc\'es (a) et (b). En ce qui concerne (c), on sait que pour toute extension galoisienne $K/k$
de groupe $G$,
on a un isomorphisme naturel (voir \cite[Lemme 14, p.~213]{ret})
$$ \ker[\br(X) \to \br(X_{K})]  \oi \ker[ H^2(G,K(X)^{\times}) \to H^2(G,{\rm Div}(X_{K}))].$$
Pour $K/k$ cyclique, la p\'eriodicit\'e de la cohomologie   des groupes finis cycliques
associe au choix d'un g\'en\'erateur de $G$ un isomorphisme entre le dernier groupe
et 
 $$\ker[ \hat{H}^{0}(G,K(X)^{\times}) \to \hat{H}^{0}(G,{\rm Div}(X_{K})]),$$
c'est-\`a-dire
 $$\ker [ k(X)^{\times}/N_{K/k}(K(X)^{\times}) \to  {\rm Div}(X)/N_{K/k}({\rm Div} X_{K})].$$
\end{proof}

\begin{rem}\label{difficilebiquad}
{\rm   Le tore $T=R^1_{L/k}\G_{m}$ associ\'e \`a une extension biquadratique $L=k(\sqrt{a},\sqrt{b})$ de $k$,
 apr\`es extension du corps de base de $k$ \`a  $K=k(\sqrt{a})$, est $K$-isomorphe \`a la $K$-vari\'et\'e d\'efinie par
 $$(x^2-by^2)(u^2-bv^2)=1.$$
 Cette $K$-vari\'et\'e est $K$-isomorphe \`a la $K$-vari\'et\'e affine d'\'equation
 $$ x^2-by^2 = z^2-bt^2 \neq 0.$$
Cette vari\'et\'e est clairement $K$-rationnelle : elle est 
 $K$-birationnelle au produit d'une droite et d'une quadrique lisse dans $\P^3_{K}$,
quadrique poss\'edant un point $K$-rationnel, et donc $K$-birationnelle \`a $\P^2_{K}$.

La proposition  s'applique. On a
$H^1(G,\pic(T^{c}_{K}))=H^1(g,\pic(T^{c}_{\ov{k}}))$ et un argument abstrait (\cite[Prop.~9.5]{ctsansuc})
montre que ce dernier groupe est isomorphe \`a $\Z/2$. Suivre cet isomorphisme n'est pas simple.
Mais surtout,  faute de bien conna\^{\i}tre
une compactification lisse explicite $T^c$ de $T$, on ne conna\^{i}t pas de fac\c on explicite
la suite exacte de $G$-modules
$$ 1 \to K(T^{c})^{\times}/K^{\times} \to {\rm Div}(T^c_{K}) \to  \pic(T^{c}_{K}) \to 0,$$
et donc on ne voit pas comment calculer l'image de l'application de groupes de cohomologie de Tate
$$\Z/2 = \hat{H}^{-1}(G,\pic(T^{c}_{K})) \to \hat{H}{^0}(G, K(T^{c})^{\times}/K^{\times}).$$

Ainsi  il n'est pas \'evident d'exhiber une fonction $f \in k(T)^{\times}$ telle que le g\'en\'erateur 
 de  $\br(T^c)/\br(k)=\Z/2$ soit donn\'e par $(K/k,f)$.}
 \end{rem}

\section{R\'eduction des tores aux tores coflasques}

On a une dualit\'e bien connue entre les $k$-tores alg\'ebriques et les
$g$-r\'eseaux de type fini, associant \`a un $k$-tore $T$ son groupe des caract\`eres $\hat{T}$ sur $\ovk$.
 Pour tout $k$-tore $T$  il existe une (plus petite) extension
finie galoisienne $K$  de $k$, telle que l'action de $g$ sur $\hat{T}$ se factorise par ${\rm{Gal}}(K/k)$.
On dit que $K/k$ d\'eploie $T$ et $\hat{T}$.

Un $k$-tore quasitrivial $P$ est un $k$-tore
dont le groupe des caract\`eres est un $g$-module de permutation.
Un tel tore est un produit de $k$-tores $R_{k_{i}/k}\G_{m}$ pour diverses
extensions s\'eparables de corps $k_{i}/k$.
Un  $k$-tore quasitrivial est $k$-isomorphe \`a un ouvert d'un espace affine $\A^d_{k}$,
et est donc une $k$-vari\'et\'e $k$-rationnelle.  Un tel $k$-tore $P$
satisfait $H^1(g,\hat{P})=0$ et $H^1(k,P)=0$ (lemme de Shapiro et th\'eor\`eme 90 de Hilbert).

Un $k$-tore $Q$ est dit coflasque (voir \cite{ctsansuc})
si pour tout sous-groupe  ouvert $h \subset g$ on a $H^1(h,\hat{Q})=0$. 

 \'Etant donn\'e un $g$-r\'eseau $M$ (groupe ab\'elien libre de type fini \'equip\'e d'une action
 continue discr\`ete de $g$), on d\'efinit
 $\cyr{X}^2_{\omega}(k,M) \subset H^2(g,M)$ comme le sous-groupe  de $H^2(g,M)$ form\'e
 des classes dont la restriction \`a tout sous-groupe ferm\'e procyclique de $g$ est nulle.
 Si l'extension finie $K/k$ de groupe $G$ d\'eploie le $g$-r\'eseau  $M$, alors
 $$\cyr{X}^2_{\omega}(k,M) \simeq  \ker [ H^2(G,M) \to \prod_{\sigma \in G} H^2((\sigma),M)].$$
Pour $M$ un $g$-module de permutation,  on a $\cyr{X}^2_{\omega}(k,M) =0$.
 
 Pour $k$ un corps de nombres, $\Omega$ l'ensemble de ses places,
  et $T$ un $k$-tore, pour   tout entier naturel $i$, on d\'efinit
$$\cyr{X}^{i}(k,T): = \ker [H^{i}(k,T) \to \prod_{v\in \Omega} H^{i}(k_{v},T)].$$
 Pour $T=P$ un tore quasitrivial, il r\'esulte du lemme de Shapiro et
 de la th\'eorie du corps de classes que l'on a $\cyr{X}^{2}(k,P)=0$.
 
\begin{prop}\label{deTaQ}
Soit $T$ un $k$-tore.

(a) Il existe une suite exacte de $k$-tores
$$ 1 \to P \to Q \to T \to 1$$
avec $P$ un $k$-tore quasitrivial et $Q$ un $k$-tore coflasque.

(b) Toute telle suite exacte est g\'en\'eriquement scind\'ee, le $k$-tore $Q$
est $k$-birationnel au produit $P \times T$. Ainsi tout $k$-tore est stablement 
$k$-birationnel \`a un $k$-tore coflasque.

(c) Si l'on se donne deux   suites exactes comme en (a)
$$ 1 \to P_{1} \to Q_{1} \to T \to 1$$
et 
$$ 1 \to P_{2} \to Q_{2} \to T \to 1,$$
alors il existe un $k$-isomorphisme de $k$-tores $P_{1} \times_{k} Q_{2} \simeq P_{2} \times_{k} Q_{1}.$

(d) La suite de caract\`eres  
$$ 0 \to \hat{T} \to \hat{Q} \to \hat{P} \to 0 $$
duale de la suite en (a)   induit un isomorphisme naturel
$$\cyr{X}^2_{\omega}(k,\hat{T}) \oi \cyr{X}^2_{\omega}(k,\hat{Q}).$$

(e) Soit $T^c$, resp.  $Q^c$, une    compactification lisse du  $k$-tore $T$,
resp. du $k$-tore $Q$. La projection $Q \to T$ induit un isomorphisme $\br(T^{c}) \oi \br(Q^{c})$.

(f) Si $k$ est un corps local, tout espace principal homog\`ene sous
un $k$-tore coflasque $Q$ est trivial. Ainsi $H^1(k,Q)=0$.

(g)  Si $k$ est un corps de nombres, on a des isomorphismes
$$\cyr{X}^1(k,Q) \oi \cyr{X}^1(k,T)$$
et
$$ \cyr{X}^1(k,Q) = H^1(k,Q).$$
\end{prop}

\begin{proof}
Pour (a), voir \cite[Prop. 1.3 (1.2.4) p.~158]{ctsansuc}. 
Par le th\'eor\`eme~90 de Hilbert et le lemme de Shapiro, tout torseur sur une
$k$-vari\'et\'e int\`egre sous un $k$-tore quasitrivial est g\'en\'eriquement scind\'e.
Comme en outre un $k$-tore quasitrivial est une $k$-vari\'et\'e $k$-rationnelle,
ceci \'etablit (b). 
Pour (c), voir  \cite[Lemma 0.6 (0.6.4)) p.~155]{ctsansuc}. 
On a la suite duale de groupes de caract\`eres
$$0 \to {\hat T} \to {\hat Q} \to {\hat P} \to 0,$$
qui est une suite exacte de modules galoisiens.
Comme $\hat{P}$ est un module de permutation,
on a $H^1(h, {\hat P})=0$ pour tout sous-groupe ouvert $h \subset g$
et $\cyr{X}^2_{\omega}(g,{\hat P})=0$. Une chasse au diagramme imm\'ediate donne alors (d).

Pour tout $k$-tore $T$, on a la formule $\br(T^c)/\br(k) \iso \cyr{X}^2_{\omega}(k,\hat{T})$  (\cite[Prop. 9.5]{ctsansuc}).
 L'\'enonc\'e (e)   r\'esulte alors de (d).
 
Prouvons (f).  Pour un  corps local $k$ et un $k$-tore $R$, les groupes finis
  $H^1(k,R)$ et $H^1(k,\hat{R})$ sont en dualit\'e  (Tate--Nakayama).
  Si le $k$-tore $R$ est coflasque, alors $H^1(k,\hat{R})=0$ et donc $H^1(k,R)=0$.

D\'emontrons (g). Pour $k$ un corps de nombres et $P$ un $k$-tore quasitrivial, on a 
 $H^1(k,P)=0$ et $\cyr{X}^2(k,P)=0$.     La suite exacte (a) donne donc $ \cyr{X}^1(k,Q) \oi  \cyr{X}^1(k,P) $.
  Par ailleurs (f) donne  $ \cyr{X}^1(k,Q) = H^1(k,Q).$
  \end{proof}

Nous aurons besoin de la proposition suivante.

\begin{prop}\protect{ \cite[Lemme 2.1]{cthask}}\label{rappelcthasko} 
Soient $T$ un $k$-tore et $E$ un $k$-espace principal homog\`ene sous $T$. 
Soit $T^c$ une $k$-compactification projective,  lisse, \  \'equi\-variante, de $T$.
Le produit contract\'e
$E^c:=E\times^TT^c$ est une $k$-compactification lisse \'equi\-variante de $E$.

(a) Il existe  
  un isomorphisme naturel
de modules galoisiens $$\pic(\ov{T}^c) \oi  \pic(\ov{E}^c).$$

(b) Il y a un isomorphisme $H^1(k,\pic(\ov{T}^c)) \oi  H^1(k,\pic(\ov{E}^c))$. 

(c) On a des suites exactes
$$0 \to \br(k) \to \br(T^c) \to H^1(k,\pic(\ov{T}^c)) \to 0,$$
$$  \br(k) \to \br(E^c) \to H^1(k,\pic(\ov{E}^c)) \to H^3(k,\G_{m}).$$
\end{prop}

\medskip

On a donc une suite exacte
$$0 \to \br(E^c)/{\rm Im}(\br(k)) \to  \br(T^c)/\br(k) \to H^3(k,\G_{m}).$$

 \section{Le tore des \'el\'ements de norme 1 d'une extension biquadratique et un tore coflasque associ\'e}

 Soit $K/k$ une extension biquadratique.
 On a donc $K=k(\sqrt{a}, \sqrt{b})$. Soit $T=R^1_{K/k}\G_{m}$ le $k$-tore noyau de la norme
 $$1 \to T \to R_{K/k}\G_{m} \to \G_{m,k} \to 1.$$
 Soit $Q$ le $k$-tore noyau d\'efini par la suite exacte
 $$1 \to Q \to \prod_{i}R_{k_{i}/k}\G_{m} \to \G_{m,k} \to 1,$$
 o\`u $k_{i}/k$ parcourt les trois sous-extensions quadratiques de $K/k$, et o\`u la fl\`eche vers $\G_{m,k}$
 est le produit des trois normes d'extensions quadratiques de $k$.

 \begin{prop}\label{coflasqueabiquad}
(a)  Le $k$-tore $Q$ est coflasque.

(b) Il existe un diagramme commutatif de groupes de type multiplicatif
\begin{equation}
 \begin{CD}
 &   &  1 &  &1 &   &1 &  &  \\
 &   &  \downarrow &  &\downarrow &   &\downarrow &  &  \\
 1 & \to  & \G_{m,k}^2 & \to &  Q  & \to & T & \to & 1 \\
&   &  \downarrow &  &\downarrow &   &\downarrow &  &  \\
 1 & \to & M & \to & \prod_{i}R_{k_{i}/k}\G_{m} & \to & R_{K/k}\G_{m} & \to & 1 \\
&   & \downarrow &  & \downarrow &   & \downarrow&   &  \\
  1 & \to & \mu_{2} & \to & \G_{m,k} & \to & \G_{m,k}  & \to & 1 \\
   &   &  \downarrow &  &\downarrow &   &\downarrow &  &  \\
   &   &  1 &  &1 &   &1 &  &  
 \end{CD}
\end{equation} 
 o\`u les deux colonnes verticales de droite sont celles d\'efinissant $Q$ et $T$,
 et la fl\`eche $\prod_{i}R_{k_{i}/k}\G_{m} \to R_{K/k}\G_{m} $ est induite par chacune
 des inclusions naturelles $R_{k_{i}/k}\G_{m} \hookrightarrow R_{K/k}\G_{m}$.
 
 (c) Notons $N_{i}=N_{k_{i}/k}(k_{i}^{\times})  \subset k^{\times}$. La suite exacte  
 $$0 \to H^1(k,Q) \to H^1(k,T) \to \br(k) \oplus \br(k)$$
 induite par la suite exacte sup\'erieure
 s'\'ecrit
 $$ 1 \to k^{\times}/N_{1}N_{2}N_{3} \to k^{\times}/N_{K/k}K^{\times} \to  \br(k) \oplus \br(k)$$
 la fl\`eche $k^{\times}/N_{1}N_{2}N_{3} \to k^{\times}/N_{K/k}K^{\times}$ \'etant induite par 
 $x \mapsto x^2$.
 
(d) On a $\cyr{X}^2_{\omega}(k,\hat{T}) \oi \cyr{X}^2_{\omega}(k,\hat{Q}) = \Z/2.$
 
 (e) Pour toute compactification lisse $E^c$ d'un espace principal homog\`ene $E$ de $Q$, on a 
 $\br(E^c)/{\rm Im}(\br(k))=\Z/2$.

 (f) Si $E$ est un espace principal homog\`ene de $T$ dont la classe dans $H^1(k,T)$ a
 une image nulle dans $\br(k) \oplus \br(k)$, i.e. si cette classe est dans l'image de
 $H^1(k,Q) \to H^1(k,T)$, alors pour toute compactification lisse $E^c$ de $E$,
  on a  $\br(E^c)/{\rm Im}(\br(k))=\Z/2$.
\end{prop}

\begin{proof}
L'\'enonc\'e (a) se voit   sur les groupes de caract\`eres.
On laisse au lecteur le soin d'\'etudier le diagramme de groupes de caract\`eres
pour identifier les termes non \'evidents dans le diagramme (b).
Une fois ceci \'etabli,  l'\'enonc\'e (c) est imm\'ediat. 
Il en est de m\^{e}me de l'\'enonc\'e (d).
Tout espace principal homog\`ene $E$  de $Q$ s'\'ecrit
 $$(x^2-ay^2)(z^2-bt^2)(u^2-abw^2)=\gamma,$$
avec $\gamma  \in k^{\times}$. Sur l'extension cyclique $k(\sqrt{a})/k$,
 une telle vari\'et\'e est $k(\sqrt{a})$-rationnelle.
 D'apr\`es la proposition \ref{cyclique}, on a donc 
 $$\br(E^c)/{\rm Im}(\br(k)) = H^1(k,\pic({\overline{E}}^c)).$$
 D'apr\`es la proposition \ref{rappelcthasko},
 on a $H^1(k,\pic({\overline{Q}}^c)) \iso H^1(k,\pic({\overline{E}}^c)).$
 Enfin, d'apr\`es (\cite[Prop. 9.5]{ctsansuc}),
  le groupe $H^1(k,\pic({\overline{Q}}^c))$ est isomorphe  \`a
$\cyr{X}^2_{\omega}(k,\hat{Q})=\Z/2.$ Ceci \'etablit (e).
Soit $E$ comme en (f), et soit $E_{1}$ un espace principal homog\`ene de $Q$
dont l'image par $H^1(k,Q) \to H^1(k,T)$ est la classe de $E$ dans ce dernier groupe.
La fibration $E_{1} \to E$
 est un torseur sour $\G_{m,k}^2$,  
  le corps des fonctions de $E_{1}$ est donc purement transcendant sur celui de $E$. 
Ceci implique que l'application 
$\br_{nr}(k(E)) \to  \  \br_{nr}(k(E_{1}))$ est un isomorphisme,
et \'etablit (f).
\end{proof}

\section{Un g\'en\'erateur explicite pour le groupe de Brauer non ramifi\'e des espaces
homog\`enes de ce tore coflasque}

\bigskip

\begin{thm}\label{generateur}
Soit $F$ un corps de caract\'eristique nulle, et soient $a,b,c \in F^{\times}$. 
Consid\'erons la $F$-vari\'et\'e $Y$ d\'efinie par  l'\'equation affine
 $$(x^2-ay^2)(z^2-bt^2)(u^2-abw^2)=c.$$
 Soit $F(Y)$ son corps des fonctions.  Soit $Z$ une $F$-compactification lisse de $Y$.

(i)  Le quotient $\br(Z)/\br(F)$ est nul si l'un des $a,b,ab$ est un carr\'e, il est \'egal \`a $\Z/2$ sinon.

(ii) L'alg\`ebre de quaternions $(x^2-ay^2,b) \in \br(F(Y))$ est non 
  ramifi\'ee sur   $Z$,
  et elle engendre le groupe $\br(Z)/\br(F)$.

\end{thm}

\begin{proof}

Si   $a$,   $b$, ou $ab$ est un carr\'e  dans $F$, alors $Y$ est $F$-rationnelle,
et $\br(F) \iso \br_{nr}(F(Y))$.

\medskip

 Soit   $A \subset F(Y)$ un anneau de valuation discr\`ete (de rang~1)
de corps des fractions $E=F(Y)$, contenant $F$. Soit $\kappa$ le corps r\'esiduel de $A$
et $$\partial_{A} : \br(F(Y)) \to H^1(\kappa,\Q/\Z)$$
l'application r\'esidu.

On a $\alpha=(x^2-ay^2,b)= (x^2-ay^2,ab) \in \br(F(Y))$.
Si $b$ ou $ab$ est un carr\'e dans $\kappa$, alors $\partial_{A}(\alpha)=0$.
Sinon, chacune des extensions $E(\sqrt{b})/E$ et $E(\sqrt{ab})/E$ est non ramifi\'ee
et inerte de degr\'e 2 au-dessus de $A$, donc les entiers $v_{A}(z^2-bt^2)$ et $v_{A}(u^2-abv^2)$ 
sont pairs. De l'\'equation on d\'eduit $v_{A}(x^2-ay^2)$ pair (ce qui est \'evident si $a$ n'est
pas un carr\'e dans $\kappa$.).
Ainsi  $\partial_{A}(\alpha)=0$ pour tout $A$, et $\alpha \in \br_{nr}(F(Y))=\br(Z)$.

\medskip

  Supposons que ni $a$, ni $b$, ni $ab$ ne sont  des carr\'es  dans $F$.
  Montrons que la classe $\alpha \in  \br_{nr}(F(Y)) \subset \br(F(Y))$
  n'appartient pas \`a $\br(F)$.
  
Consid\'erons la projection de $Y$ vers l'espace affine $\A^4$
de coordonn\'ees $(x,y,z,t)$. La fibre g\'en\'erique est la conique
$$u^2-abw^2= c(x^2-ay^2)(z^2-bt^2).$$
D'apr\`es Witt (\cite{witt}, Satz, S. 465) (pour la g\'en\'eralisation par Amitsur, voir \cite[Thm. 5.4.1]{gilleszamuely}),
le noyau de $\br(F(\A^4)) \to \br(F(Y))$ est d'ordre au plus 2, engendr\'e par
la classe de l'alg\`ebre de quaternions $$\beta=(c(x^2-ay^2)(z^2-bt^2),ab) \in \br(F(\A^4)).$$

Pour montrer que $\alpha \in     \br(F(Y))$
  n'appartient pas \`a $\br(F)$,
il suffit donc de voir que  
la classe $(x^2-ay^2,b) \in \br(F(\A^4))$ n'appartient pas au sous-groupe engendr\'e par
$\br(F)$ et
  $$\beta=(c(x^2-ay^2)(z^2-bt^2),ab)=(c,ab)+(x^2-ay^2,b)+(z^2-bt^2,a).$$
Les hypoth\`eses assurent que $x^2-ay^2=0$ et $z^2-bt^2=0$
sont int\`egres et que $b$, resp. $a$, n'est pas un carr\'e  dans le corps r\'esiduel
de $x^2-ay^2=0$, resp. $z^2-bt^2=0$.

Le r\'esidu de $(x^2-ay^2,b) \in \br(F(\A^4))$ en $x^2-ay^2=0$ est la classe de $b$, qui est non nulle. Donc $$(x^2-ay^2,b) \notin  \br(F) \subset \br(F(\A^4)).$$

 Le r\'esidu  de  $(x^2-ay^2,b)$ en $z^2-bt^2=0$ est  nul, mais le r\'esidu de $\beta$
en $z^2-bt^2=0$ est $a$, qui est non nul. On a donc $$(x^2-ay^2,b) - \beta \notin  \br(F) \subset \br(F(\A^4)).$$

La $F$-vari\'et\'e $Y$ est un espace principal homog\`ene sous le $F$-tore $Q$ d\'efini par
 $$(x^2-ay^2)(z^2-bt^2)(u^2-abw^2)=1.$$
D'apr\`es la proposition \ref{coflasqueabiquad} (e),  on a $\br_{nr}(F(Y))/{\rm Im}(\br(F))=\Z/2$.
Ceci conclut la d\'emonstration.
\end{proof}

 \section{Une application num\'erique}
 
 \begin{prop}
Soient $k$ un corps de nombres et $a,b,c \in k^{\times}$.
Consid\'erons la $k$-vari\'et\'e $Y$ d\'efinie par  l'\'equation affine
 $$(x^2-ay^2)(z^2-bt^2)(u^2-abw^2)=c.$$
 Soit $Z$ une $k$-compactification lisse de $Y$.

(a) Pour toute place $v$ de $k$, on a $Y(k_{v}) \neq \emptyset$.

(b) Soit $\alpha \in \br(Z) \subset \br(Y)$ d\'efini par $(x^2-ay^2,b)$.
Cet \'el\'ement engendre $\br(Z)/\br(k)$.

(c) S'il existe une place $v$ telle qu'aucun des $a,b,ab$ ne soit un carr\'e dans $k_{v}$,
alors $Y(k) \neq \emptyset$.

(d) Supposons    qu'en toute place $v$ de $k$ l'un au moins des $a,b$ ou $ab$ est un carr\'e
dans le compl\'et\'e $k_{v}$, alors pour toute famille ad\'elique $\{P_{v}\} \in \prod_{v} Z(k_{v})$,
on~a $$ \sum_{v} \beta(P_{v}) =  \sum_{v, a \in k_{v}^{\times 2}} (c,b)_{v}  =   \sum_{v, a \notin k_{v}^{\times 2}} (c,b)_{v}  \in \Z/2 \subset  \Q/\Z .$$ La nullit\'e de cet \'el\'ement est alors une condition
n\'ecessaire et suffisante pour l'existence d'un $k$-point sur $Z$ et $Y$.
 \end{prop}
 
 \begin{proof}
 
 (a) Le $k$-tore $Q$ d\'efini par
 $$(x^2-ay^2)(z^2-bt^2)(u^2-abw^2)=1$$
  est coflasque (Prop. \ref{coflasqueabiquad}
  (a)). La $k$-vari\'et\'e $Y$ est   un espace principal homog\`ene de $Q$.
Par  la proposition
  \ref{deTaQ} (f), on a donc $H^1(k_{v},Q)=0$ pour toute place $v$ de $k$,
ce qui donne (a).

L'\'enonc\'e (b) a fait l'objet du th\'eor\`eme \ref{generateur}.

(c) Si en une place $v$ l'extension $k_{v}(\sqrt{a},\sqrt{b})$ est biquadratique,
alors l'\'el\'ement $\alpha_{v }\in \br(k_{v}(Y))$ engendre $\br(Z_{k_{v}})/\br(k_{v})=\Z/2$ d'apr\`es le th\'eor\`eme 
\ref{generateur}. Il r\'esulte alors de \cite[Cor. 1, p.~217]{ret}
que $\alpha_{v} \in \br(Z_{k_{v}})$ prend deux valeurs distinctes sur $Y(k_{v})$.
Notons que $Y(k_{v}) \neq \emptyset$ implique que $Y_{k_{v}}$ est
$k_{v}$-isomorphe au $k_{v}$-tore $Q_{k_{v}}$. Comme $\alpha$
est d'exposant 2,
ceci implique qu'il existe un
ad\`ele $P_{v} \in Z(\A_{k})$ tel que 
$$\sum_{v} \alpha(P_{v})=0.$$
Comme $\alpha$ engendre $\br(Z)/\br(k)$, il n'y a pas  d'obstruction de Brauer--Manin
pour la compactification lisse $Z$ de l'espace homog\`ene $Y$ sous le $k$-tore $Q$.
D'apr\`es Sansuc  \cite[Cor. 8.7]{sansuccrelle}), ceci implique $Y(k)\neq \emptyset$.

D\'emontrons (d).  Soit $v$ une place de $k$.
On a l'inclusion $k_{v} \subset k_{v}(Y)$.
Si $b \in k_{v}^{\times 2}$, alors $\alpha_{v}=0 \in \br(k_{v}(Y))$.
Si $ab  \in k_{v}^{\times 2}$,  alors $\alpha_{v}=0 \in \br(k_{v}(Y))$.
Si $a \in k_{v}^{\times 2}$, alors $\alpha_{v}=(c,b)_{v} \in \br(k_{v}(Y))$ est l'image de $(c,b)_{v} \in \br(k_{v}).$ Ceci donne la premi\`ere \'egalit\'e.
La seconde \'egalit\'e provient de la loi de r\'eciprocit\'e de la th\'eorie du corps de classes global.
Le dernier \'enonc\'e r\'esulte alors de (b), et du r\'esultat de Sansuc  \cite[Cor. 8.7]{sansuccrelle}).
\end{proof}
 
  \begin{rem}
  \label{Emorysuresh}
{\rm Voici une d\'emonstration directe du point (a), communiqu\'ee par V.~Suresh.
Pour toute place $v$ et  tout triplet $(a,b,c)$ d'\'el\'ements
de $k_{v}^{\times}$, et toute place $v$, l'un des trois symboles $(a,c)_{v}, (b,c)_{v},  (ab,c)_{v}$
est nul. En effet leur somme est nulle, et ils valent soit $0$ soit $1/2$.}
\end{rem}

\medskip

Il est ainsi facile de fabriquer des contre-exemples au principe de Hasse
classique pour des vari\'et\'es donn\'ees par $Y$.   
Sur le corps $\Q$, on sait bien qu'en toute place $v$ soit $13$, soit $17$, soit $13\times 17$
est un carr\'e. On prend $a=17$, $b=13$.
On cherche $c$ tel que
$$\sum_{v, 17 \notin k_{v}^{\times 2}} (c,13)_{v} \neq 0.$$
Ceci  donne
$$\sum_{p \neq 2,13, 17; \hskip2mm 17  \notin \F_{p}^{\times 2}} (c,13)_{p} \neq 0.$$
(Noter que $17$ est un carr\'e dans $\Q_{2}$
et dans $\Q_{13}$). Par la loi de r\'eciprocit\'e, la condition $17$ non carr\'e dans le corps fini  $\F_{p}$ se traduit :
$p$ est  non carr\'e dans $\F_{17}$.

Si l'on prend $c=l$ un nombre premier non carr\'e dans $\F_{17}$ et non carr\'e dans $\F_{13}$,
la somme se r\'eduit \`a $(l,13)_{l} \neq 0$. Le nombre premier $l=5$ convient.
Ainsi :
 $$(x^2-13y^2)(z^2-17t^2)(u^2-221w^2)=5$$
 est un contre-exemple au principe de Hasse, 
 et  la combinaison des propositions  \ref{deTaQ} et \ref{coflasqueabiquad} 
 montre que
 $$N_{\Q(\sqrt{13},\sqrt{17})/\Q}(\Xi)=5^2$$
est   un
  contre-exemple au principe de Hasse.

Pour des calculs similaires, on consultera Sansuc \cite{sansuc}.

\end{document}